\documentclass[11pt,a4paper]{article}
\usepackage[text={160mm,254mm},centering]{geometry}
\usepackage{parskip,hyperref,amsthm,amsmath,amssymb,microtype}

\renewcommand{\mathbb}{\mathbf}
\usepackage{microtype}

\newtheorem*{theorem*}{Theorem}
\newtheorem{theorem}{Theorem}
\newtheorem*{question}{Question}
\newtheorem{proposition}{Proposition}
\theoremstyle{remark}
\newtheorem*{remark}{Remarks}
\renewcommand{\section}{\subsection}

\newcommand{\coloneq}{:=}
\date{}
\title{On sums of Betti numbers of affine varieties}

\author{Dingxin Zhang\footnote{Partially supported by the National Key Research and Development
Program of China (No. 2022YFA1007100).}}
\begin{document}
\maketitle

\begin{abstract}
\noindent
We show that if \(V\) is a subvariety of the affine \(N\)-space
defined by polynomials of degree at most \(d\), then the sum of its
\(\ell\)-adic Betti numbers does not exceed
\(2(N+1)^{2N+1}(d+1)^{N}\).  This answers a question of N.~Katz.
\end{abstract}

\section{Introduction}
Let \(V\) be a finite type, separated scheme over an algebraically
closed field \(k\).  Katz \cite[Theorem~5]{k2} proved the existence of
a constant \(M(V/k)\) such that for any prime \(\ell\) invertible in
\(k\), the following inequality holds:
\begin{equation*}
B(V)_{(\ell)} \coloneq \sum_{i} \dim \mathrm{H}^{i}(V;\mathbb{Q}_{\ell}) \leqslant M(V/k),
\end{equation*}
where \(\mathrm{H}^{i}(V;\mathbb{Q}_{\ell})\) denotes the
\(\ell\)-adic cohomology.  When \(k\) has positive characteristic, it
is conjectured that \(B(V)_{(\ell)}\) is independent of \(\ell\).
Katz's theorem shows that \(B(V)_{(\ell)}\) possesses a uniform upper
bound as \(\ell\) varies, thus can be interpreted as a weak form of
the conjectured independence.

Katz's proof relies on selecting an alteration of \(V\), which leaves
the constant \(M(V/k)\) implicit.  He then posed the following
question \cite[p.~36]{k2}, paraphrased here:

\begin{question}[Katz]
If \(V \subset \mathbb{A}^{N}_{k}\) is defined by polynomials of
degree at most \(d\), can an explicit upper bound for \(M(V/k)\) be
provided in terms of \(d\) and the number of the defining polynomials?
\end{question}

If \( V \) is \emph{smooth and connected}, Katz \cite[Corollary 2]{k2}
proved that
\begin{equation}\label{eq:katz}\tag{K}
B(V)_{(\ell)} \leqslant 3 \times 2^r \times (r+1+rd)^N,\footnote{Katz did
  not focus on obtaining the most optimal bound, making this
  inequality somewhat crude. His method actually yields a sharper
  bound: \(2^r(rd + r + 2)^N\).}
\end{equation}
Thus, the main challenge lies in addressing singular varieties.

When \( k = \mathbb{C} \), Milnor \cite[Corollary 1]{m} proved
\begin{equation}
\label{eq:milnor}\tag{MOT}
\sum_{i} \dim \mathrm{H}^{i}(V^{\mathrm{an}};\mathbb{Q}) \leqslant d(2d-1)^{2N-1}
\end{equation}
for the singular cohomology of the complex analytic space
\(V^{\mathrm{an}}\) associated with \(V\).  Similar bounds were
established by Oleinik \cite{o} and Thom \cite{t}.  Since the
dimensions of the \(\ell\)-adic cohomology of \(V\) and the singular
cohomology of \(V^{\mathrm{an}}\) coincide, this resolves the question
for \(k = \mathbb{C}\).  However, Milnor's proof relies on Morse
theory, which does not extend to positive characteristic.


In this brief note, we answer Katz's question with the following theorems:

\begin{theorem}\label{theorem:with-r}
For any algebraically closed field \(k\), any closed subvariety \(V\)
of \(\mathbb{A}^{N}_{k}\) cut out by \(r\geqslant 1\) polynomials of
degree at most \(d\), and any prime \(\ell\) invertible in \(k\), we
have
\begin{equation*}
B(V)_{(\ell)} \leqslant 2r^{r}(rd+3)^{N}.
\end{equation*}
\end{theorem}

If we do not want to specify the number \(r\) of defining polynomials
of \(V\), we have the following:

\begin{theorem}\label{theorem:main}
For any algebraically closed field \(k\), any closed subvariety \(V\)
of \(\mathbb{A}^{N}_{k}\) defined by polynomials of degree at most
\(d\), and any prime \(\ell\) invertible in \(k\), we have
\begin{equation*}
B(V)_{(\ell)} \leqslant 2(N+1)^{2N+1}(d+1)^{N}.
\end{equation*}
\end{theorem}

In fact, Theorem~\ref{theorem:main} is a consequence of
Theorem~\ref{theorem:with-r}.  We may assume \(N \geqslant 2\) since
when \(N=1\) we trivially have \(B(V)_{\ell}\leqslant d\).  By a
classical theorem of Kronecker (cf.~\cite{p} or \cite[\S3]{crw}),
\(V\) can always be set-theoretically cut out by at most \(N+1\)
nonzero polynomials of degree \(\leqslant d\).  Hence, by taking
\(r = N+1\) in Theorem~\ref{theorem:with-r}, we get
\begin{equation*}
B(V)_{(\ell)} \leqslant 2 \times \left[ (N+1)d+3 \right]^{N} \times (N+1)^{N+1}
\leqslant 2\times (N+1)^{2N+1}\times (d+1)^{N}.
\end{equation*}

\begin{remark}
(i) When \(k = \mathbb{C}\), Theorem~\ref{theorem:main} sharpens the
classical Milnor--Oleinik--Thom upper bound \eqref{eq:milnor} if \(d\)
is large compared to \(N^{2}\).

(ii) Consider the number
\begin{equation*}
B(N,d) \coloneq \sup\left\{ B(V)_{(\ell)} :
\begin{array}{l}
  V = \{f_{1}=\cdots = f_{r} = 0\} \subset \mathbb{A}_{k}^{N} \\
   \text{for some }r, \text{ with }\deg f_{i} \leqslant d
\end{array}
\right\}.
\end{equation*}
If we treat \(N\) as fixed and \(d\) as a variable, then
Theorem~\ref{theorem:main} implies that \(B(N,d) \ll_{N} d^{N}\).
On the other hand, we have \(B(N,d)\geqslant d^{N}\).  Indeed, if
\(V\) is a transverse intersection of \(N\) sufficiently general
degree \(d\) polynomials in \(N\) variables, then \(V\) is a finite
set comprised of \(d^{N}\) points, and \(B(V)_{(\ell)} = d^{N}\).  Thus
\(B(N,d)\) and \(d^{N}\) have the same asymptotic order as
\(d \to \infty\): \(B(N,d) \asymp_{N} d^{N}\).

(iii) Let \(R\) be any commutative ring, and let \(\mathcal{Z}\) be a
finite type separated scheme over \(R\).  For any geometric point
\(x\colon \operatorname{Spec}k \to \operatorname{Spec}R\), Katz showed
that
\[
B(\mathcal{Z} \otimes_{R,x} k) \leqslant M(\mathcal{Z} \otimes_{R,x} k/k).
\]
However, these upper bounds may depend on the specific geometric point
\(x\).  In contrast, if
\[
\mathcal{Z} = \operatorname{Spec} R[x_{1}, \ldots, x_{N}]/(f_{1}, \ldots, f_{r}),
\]
with \(\deg f_{i} \leqslant d\), then the bounds established in
Theorems~\ref{theorem:with-r} and \ref{theorem:main} are given by
explicit constants that apply uniformly to every geometric point \(x\)
of \(\operatorname{Spec}R\).  Thus, these bounds have the advantage of
being \emph{uniform in \(k\)}.

(iv) Although we have established that \(d^{N}\) is the correct
asymptotic order of \(B(N,d)\) as \(d \to \infty\), the coefficient
\(2(N+1)^{2N+1}\) of \(d^{N}\) in Theorem~\ref{theorem:main} is far
from optimal.  There should be considerable room for improvement.
\end{remark}

For complete intersections, a slightly sharper bound can be
established.

\begin{theorem}\label{theorem:complete-intersection}
Let \(k\) be an algebraically closed field.  Suppose
\(V \subset \mathbb{A}^{N}_{k}\) is the common zero locus of \(r\)
polynomials of degree at most \(d\).  If \(V\) has at worst local
complete intersection singularities, e.g., if \(\dim V = N-r\), then
\begin{equation*}
B(V)_{(\ell)} \leqslant 2^{r} (rd+r+2)^{N}.
\end{equation*}
\end{theorem}

\section{The proofs}

In the following, we fix an algebraically closed field \(k\) and a
prime \(\ell\) invertible in \(k\).  All schemes are defined over
\(k\).  We shall write \(B(V)\) instead of \(B(V)_{(\ell)}\).

Suppose \(B(N,r,d)\) is a positive integer satisfying the following
property.
\begin{quote}
\textit{If \(V\) is a closed subscheme of \(\mathbb{A}^{N}\) defined
  by \(r\) polynomials \(f_{1},\ldots,f_{r}\) with
  \(\deg f_{i} \leqslant d\), then \(B(V) \leqslant B(N,r,d)\).}
\end{quote}

It is not immediately evident that a finite \(B(N,r,d)\) even
exists.  This does not follow directly from Deligne's theorem on the
constructibility of the direct image of constructible sheaves, since
the non-proper direct image does not generally commute with base
change.  Some additional argument is indeed required, though it is not
difficult.  We will not dwell on this point, as it will follow from
the argument presented below.

Suppose \(E(N,r,d)\) is a positive integer satisfying the
following property:
\begin{quote}
\textit{If \(V\) is a subvariety of \(\mathbb{A}^{N}\) defined by
  nonzero polynomials \(f_{1},\ldots,f_{r}\) with
  \(\deg f_{i} \leqslant d\), then
  \(|\chi(V;\mathbb{Q}_{\ell})| \leqslant E(N,r,d)\).}
\end{quote}
Here \(\chi(V;\mathbb{Q}_{\ell})\) is the Euler characteristic
\(\chi(V;\mathbb{Q}_{\ell})=\sum_{i}(-1)^{i}\dim\mathrm{H}^{i}(V;\mathbb{Q}_{\ell})\).
By \cite{l}, the Euler characteristic equals the compactly supported
Euler characteristic:
\[
\chi(V;\mathbb{Q}_{\ell})=\sum_{i}(-1)^{i}\dim\mathrm{H}^{i}_{c}(V;\mathbb{Q}_{\ell}).
\]
Hence, by \cite[Theorem~5.27]{as} (and the comment at the beginning of
\cite[p.~30]{k2}) we can take
\begin{equation}\label{eq:as}\tag{AS}
E(N,r,d) = 2^{r} \times (r+1+rd)^{N}.
\end{equation}

\subsubsection*{Upper bound for local complete intersections: Katz's method}

Let \(f_{1},\ldots,f_{r} \in k[x_{1},\ldots,x_{N}]\) be polynomials of
degree \(\leqslant d\).  Consider the affine scheme
\[V\coloneq\operatorname{Spec}k[x_{1},\ldots,x_{N}]/(f_{1},\ldots,f_{r}).\]
We assume that \(V\) is a \emph{set-theoretic local complete
  intersection}, meaning there is a closed subscheme \(W\) of
\(\mathbb{A}^{N}\) that is a local complete intersection, with
\(W^{\mathrm{red}} = V^{\mathrm{red}}\).  This is the case, for
example, if \(\dim V=N-r\), or if \(V\) is smooth.  We seek an
explicit upper bound for \(B(V)\).

\begin{proposition}
\label{proposition:hyper}
In the above situation, we have
\[
B(V) \leqslant E(N,r,d) + 2\sum_{i=1}^{\dim V} E(N-i,r,d).
\]
In particular, we can take \(B(N,1,d)\) to be
\(E(N,1,d)+2\sum_{i=1}^{N} E(i,1,d)\).
\end{proposition}

The proof of Proposition~\ref{proposition:hyper} closely follows
Katz's proof of inequality \eqref{eq:katz}.  Instead of using the
standard weak Lefschetz theorem for smooth varieties as Katz did, we
apply the following weak Lefschetz theorem for perverse sheaves, due
to Deligne.

\begin{theorem*}[Deligne {\cite[Corollary~A.5]{k1}}]
Let \(\pi\colon X \to \mathbb{P}^{N}\) be a quasi-finite morphism to a projective space.
Let \(\mathcal{P}\) be a perverse sheaf on \(X\). Then for a sufficiently general
hyperplane \(A\), the restriction morphism
\begin{equation*}
\mathrm{H}^i(X;\mathcal{P}) \to \mathrm{H}^i(\pi^{-1}A; \mathcal{P}|_{\pi^{-1}A})
\end{equation*}
is injective if \(i = -1\), and bijective if \(i < -1\).
\end{theorem*}

We now use Deligne's theorem and Katz's original Euler characteristic
argument \cite[p.~33]{k2} to prove
Proposition~\ref{proposition:hyper}.

\begin{proof}[Proof of Proposition~\ref{proposition:hyper}]
We proceed by induction on the dimension of \(V\).  In the base case
where \(\dim V = 0\), the result is trivial.  Now assume
\(\dim V > 0\).  Since \(V\) is a set-theoretic local complete
intersection, the shifted constant sheaf
\(\mathbb{Q}_{\ell,V}[\dim V]\) is a perverse sheaf on \(V\), see
e.g., \cite[Lemma~III.6.5]{kw}.  Now apply Deligne's theorem by
\begin{itemize}
\item taking \(X = V\),
\item letting \(\pi\) be the composition of the inclusion map
\(V \hookrightarrow \mathbb{A}^{N}\) and the standard embedding
\(\mathbb{A}^{N} \hookrightarrow \mathbb{P}^{N}\), and
\item setting \(\mathcal{P} = \mathbb{Q}_{\ell,V}[\dim V]\).
\end{itemize}
Deligne's theorem then implies that for a general affine hyperplane
\(A \subset \mathbb{A}^{N}\), the restriction map
\begin{equation*}
\mathrm{H}^{i}(V;\mathbb{Q}_{\ell}) \to \mathrm{H}^{i}(A \cap V; \mathbb{Q}_{\ell})
\end{equation*}
is injective when \(i = \dim V-1\), and bijective if \(i < \dim V-1\).
Ergo,
\begin{align}
  &\;\dim \mathrm{H}^{\dim V}(V;\mathbb{Q}_{\ell}) \nonumber \\
  =&\; (-1)^{\dim V}\chi(V;\mathbb{Q}_{\ell}) + \dim \mathrm{H}^{\dim V-1}(V;\mathbb{Q}_{\ell}) - \dim \mathrm{H}^{\dim V-2}(V;\mathbb{Q}_{\ell}) + \cdots \nonumber \\
  \leqslant&\; (-1)^{\dim V}\chi(V;\mathbb{Q}_{\ell})  + \dim \mathrm{H}^{\dim V-1}(A \cap V; \mathbb{Q}_{\ell}) - \dim \mathrm{H}^{\dim V-2}(A\cap V;\mathbb{Q}_{\ell}) + \cdots \nonumber \\
  =&\; (-1)^{\dim V}\chi(V;\mathbb{Q}_{\ell}) + (-1)^{\dim V-1}\chi(A \cap V;\mathbb{Q}_{\ell}) \nonumber\\
  \leqslant&\; E(N,r,d) + E(N-1,r,d). \label{eq:middle}
\end{align}
Here we have used the Artin vanishing theorem, which asserts that for a
finite type affine scheme \(V\) over \(k\),
\(\mathrm{H}^{i}(V;\mathbb{Q}_{\ell})=0\) unless
\(0\leqslant i \leqslant \dim V\).

Note that \(A \cap V\) is a closed subscheme of the
\((N-1)\)-dimensional affine space \(A\), defined by polynomials of
degree \(\leqslant d\).  Since \(A\) is generic and \(V\) is a
set-theoretic local complete intersection, it follows that
\(A \cap V\) also remains a set-theoretic local complete intersection.
By inductive hypothesis, we have
\begin{equation}
\label{eq:others}
B(A \cap V) \leqslant E(N-1,r,d) + 2 \sum_{i=1}^{\dim V-1}E(N-1-i,r,d).
\end{equation}
Therefore,
\begin{align*}
  B(V)
  &\leqslant \dim \mathrm{H}^{\dim V}(V;\mathbb{Q}_{\ell}) + B(A \cap V)
  & \text{(Deligne's theorem)}\\
  &\leqslant E(N,r,d) + E(N-1,r,d)  \\
  &\quad + E(N-1,r,d) + 2 \sum_{1\leq i\leq \dim V-1} E(N-1-i,r,d)
  & \text{(By \eqref{eq:middle} and \eqref{eq:others})}\\
  &= E(N,r,d) + 2\sum_{i=1}^{\dim V} E(N-i,r,d).
\end{align*}
This completes the proof.
\end{proof}

\begin{proof}[Proof of Theorem~\ref{theorem:complete-intersection}]
We take \(E(N,r,d)\) as in \eqref{eq:as}.  When \(N = \dim V\), we
have \(V = \mathbb{A}^{N}\), and the result is trivial.  When
\(N = 1\), and \(\dim V = 0\), the result is equally trivial since we
have \(B(V) \leq d\).  Assume now \(N \geq 2\), and \(\dim V < N\).
Then by Proposition~\ref{proposition:hyper}, we have
\begin{align}
  B(V)
  &\leqslant 2^{r}\left[ (rd+r+1)^{N} + 2\sum_{i=1}^{\dim V}(rd+r+1)^{N-i} \right]
    \nonumber \\
  &\leqslant 2^{r}(rd+r+2)^{N} \label{eq:complete-intersection}
\end{align}
thanks to the binomial theorem.  This completes the proof.
\end{proof}

\subsubsection*{Upper bounds in general}

Now suppose we have inductively constructed \(B(N,1,d)\),
\(B(N,2,d)\), \(\ldots\), up to \(B(N,r-1,d)\), and these numbers are
all finite.  We proceed to deal with varieties defined by \(r\)
equations.

Suppose \(f_{1},\ldots,f_{r} \in k[x_{1},\ldots,x_{N}]\),
\(\deg f_{i}\leqslant d\).  Define:
\begin{itemize}
\item \(F_{i} \coloneq \{f_{i} = 0\}\),
\item \(W \coloneq \bigcup F_{i}\),
\item for each \(J \subset \{1,\ldots,r\}\),
\(F_{J} \coloneq \bigcap_{j\in J}F_{j}\),
\item \(V \coloneq \{f_{1}=\cdots=f_{r}=0\} = F_{\{1,\ldots,r\}}\).
\end{itemize}
We will bound \(B(V)\).

\begin{proposition}\label{proposition:general}
In the situation above, we have
\begin{equation*}
B(V) \leqslant B(N,1,rd) + \sum_{i=1}^{r-1}\binom{r}{i} B(N,i,d).
\end{equation*}
In particular, we can take \(B(N,r,d)\) to be
\(B(N,1,rd)+\sum_{i=1}^{r-1}\binom{r}{i}B(N,i,d)\).
\end{proposition}

\begin{proof}
There is a Mayer--Vietoris spectral sequence for the finite closed
covering \(\bigcup_{i=1}^{r} F_{i}\) of \(W\):
\begin{equation*}
E_{1}^{p,q} =\bigoplus_{\operatorname{Card}J=p+1} \mathrm{H}^{q}(F_{J};\mathbb{Q}_{\ell})
\Rightarrow \mathrm{H}^{p+q}(W;\mathbb{Q}_{\ell}).
\end{equation*}
We have \(E_{1}^{r-1,q}=\mathrm{H}^{q}(V;\mathbb{Q}_{\ell})\).  Because
\(E_{\infty}^{r-1,q}\) is a subquotient of
\(\mathrm{H}^{q+r-1}(W;\mathbb{Q}_{\ell})\), we have
\begin{equation*}
\sum_{q}\dim E_{\infty}^{r-1,q} \leqslant B(W).
\end{equation*}
For each \(q\) and each \(i\), \(E_{i}^{r-1,q}\) appears at the
rightmost column of the \(E_{i}\)-page of the spectral sequence, i.e.,
\(E_{i}^{p,q} = 0\) for \(p \geqslant r\).  Therefore, we have
\(E_{i+1}^{r-1,q} = E_{i}^{r-1,q}/d_{i}(E_{i}^{r-1-i,q+i-1})\), where
\(d_{i}\colon E^{p,q}_{i} \to E^{p+i,q-i+1}_{i}\) is the differential
of the spectral sequence.  It follows that
\begin{align}
  B(W) &= \sum_{p,q} \dim E_{\infty}^{p,q}
         \geqslant \sum_{q} \dim E_{\infty}^{r-1,q} \nonumber \\
       & = \sum_{q} \dim E_{1}^{r-1,q} - \sum_{q}\sum_{i=1}^{r-1}\dim d_{i}(E_{i}^{r-1-i,q+i-1}) \nonumber \\
       &\geqslant B(V) - \sum_{i=1}^{r-1} \sum_{q} \dim E_{1}^{r-1-i,q+i-1} \label{eq:1}
\end{align}
The last inequality holds because for any \(i\geqslant 1\),
\(E_{i}^{p,q}\) is a subquotient of \(E_{1}^{p,q}\).

For each \(p \geqslant 0\), \(E_{1}^{p,q}\) is a direct sum of
\(\mathrm{H}^{q}(F_{J})\) with \(\operatorname{Card}J=p+1\).  There
are \(\binom{r}{p+1}\) such summands.  Since
\(F_{J}\subset \mathbb{A}^{N}\) is cut out by \(p+1\) polynomials of
degree \(\leqslant d\), we have
\begin{align*}
  \sum_{q}\dim E_{1}^{p,q} &= \sum_{q} \bigoplus_{\operatorname{Card}J=p+1} \mathrm{H}^{p}(F_{J};\mathbb{Q}_{\ell}) \\
  &\leqslant \binom{r}{p+1} B(N,p+1,d).
\end{align*}
Hence,
\begin{equation}\label{eq:2}
\sum_{i=1}^{r-1}\sum_{q} \dim E_{1}^{r-1-i,q+i-1} \leqslant \sum_{i=1}^{r-1}\binom{r}{r-i} B(N,r-i,d).
\end{equation}
Since \(W\) is a hypersurface cut out by \(f_{1}\cdots f_{r}\), a
polynomial of degree \(\leqslant rd\), we conclude from
Proposition~\ref{proposition:hyper} that \(B(W) \leqslant B(N,1,rd)\).
Thus Equations \eqref{eq:1} and \eqref{eq:2} imply that
\begin{equation*}
B(V) \leqslant B(N,1,rd) + \sum_{i=1}^{r-1} \binom{r}{i}B(N,i,d).
\end{equation*}
This completes the proof.
\end{proof}

\begin{proof}[Proof of Theorem~\ref{theorem:with-r}]
In view of Proposition~\ref{proposition:hyper} and
Proposition~\ref{proposition:general}, if the numbers \(B(N,r,d)\) are
defined inductively as
\begin{itemize}
\item \(B(N,1,d)=E(N,d) + 2\sum_{i=1}^{N-1}E(i,d)\),
\item \(B(N,r,d) = B(N,1,rd) + \sum_{i=1}^{r-1}\binom{r}{i}B(N,i,d)\),
\end{itemize}
where \(E(N,d)\) is given by \eqref{eq:as}, then
\(B(V) \leqslant B(N,r,d)\).  Let us prove
\(B(N,r,d) \leqslant 2r^{r}(rd+3)^{N}\) by induction.  For the base
case, we invoke Equation~\eqref{eq:complete-intersection} with
\(r=1\).  When \(r > 1\), we have
\begin{align*}
  B(N,r,d)
  &\leqslant 2(rd+3)^{N} + \sum_{i=1}^{r-1} \binom{r}{i} \times 2i^{i}(id+3)^{N}
  & \text{(inductive hypothesis)}\\
  &\leqslant 2(rd+3)^{N} \left[ 1+\sum_{i=1}^{r-1} \binom{r}{i} (r-1)^{i} \right]
  & (\text{since } i \leqslant r-1 < r)\\
  &\leqslant 2 \times (rd+3)^{N} \times r^{r}
  & \text{(binomial theorem).}
\end{align*}
This completes the proof of Theorem~\ref{theorem:with-r}.
\end{proof}

\noindent
\textit{Acknowledgment.}  To be added.

\medskip\noindent
\textsc{C629 Shuangqing Complex Building A, Tsinghua University, Beijing, China.}

\noindent
\textit{Email}: \texttt{zhangdingxin13@gmail.com}
\end{document}